\documentclass[letterpaper, 10 pt, conference, onecolumn]{ieeeconf}  % Comment this line out
\IEEEoverridecommandlockouts                              % This command is only
\usepackage{cite}
\usepackage{amsmath}
\usepackage{graphics}
\usepackage{graphicx}
\usepackage{amssymb}
\usepackage{url}
\usepackage{comment}
\usepackage{caption}
\usepackage{multirow}

\def\ba{\begin{array}}
	\def\ea{\end{array}}
\newcommand{\beq}{\begin{equation}}
\newcommand{\eeq}{\end{equation}}
\newcommand{\bq}{\begin{eqnarray}}
\newcommand{\eq}{\end{eqnarray}}
\newcommand{\bqn}{\begin{eqnarray*}}
	\newcommand{\eqn}{\end{eqnarray*}}
\newcommand{\bee}{\begin{enumerate}}
	\newcommand{\eee}{\end{enumerate}}
\newcommand{\bi}{\begin{itemize}}
	\newcommand{\ei}{\end{itemize}}

\newtheorem{definition}{\textbf{Definition}}

\newtheorem{corollary}{\textbf{Corollary}}
\newtheorem{lemma}{\textbf{Lemma}}
\newtheorem{remark}{\textbf{Remark}}
\newtheorem{theorem}{\textbf{Theorem}}

\usepackage{ifthen}
\newboolean{showcomments}
\setboolean{showcomments}{true}
\usepackage[usenames,dvipsnames]{xcolor}
\usepackage{todonotes}
\usepackage{soul}

\definecolor{bleudefrance}{rgb}{0.19, 0.55, 0.91}
\definecolor{ao(english)}{rgb}{0.0, 0.5, 0.0}

\newcommand{\addcite}[0]{\ifthenelse{\boolean{showcomments}}
{\textcolor{purple}{(add cite(s)) }}{}}%

\newcommand{\enrique}[1]{  \ifthenelse{\boolean{showcomments}}
{\todo[inline,color=bleudefrance]{Enrique: #1}}{}}
\newcommand{\emmargin}[1]{\ifthenelse{\boolean{showcomments}}{\marginpar{\color{bleudefrance}\tiny EM: #1}}{}}
\newcommand{\emumargin}[1]{  \ifthenelse{\boolean{showcomments}}
{\todo[color=bleudefrance,size=\tiny]{EM: #1}} {}}

\makeatletter
\if@todonotes@disabled

\else

\fi
\makeatother

\newboolean{showedits}
\setboolean{showedits}{true}
\usepackage[xcolor=bleudefrance]{changes}
\definechangesauthor[color=bleudefrance]{EM}
\newcommand{\aem}[1]{
\ifthenelse{\boolean{showedits}}
{\added[id=EM]{#1}}
{#1}
}
\newcommand{\chem}[2]{
\ifthenelse{\boolean{showedits}}
{\replaced[id=EM]{#1}{#2}}
{#1}
}
\newcommand{\dem}[1]{
\ifthenelse{\boolean{showedits}}
{\deleted[id=EM]{#1}}
{}
}
\title{\LARGE \bf
The Role of Strategic Load Participants in Two-Stage Settlement Electricity Markets
}

\author{Pengcheng You, Dennice F. Gayme, and Enrique Mallada% <-this % stops a space
\thanks{This work was supported by ARO through contract W911NF-17-1-0092, US DoE EERE award DE-EE0008006, NSF through grants CNS 1544771, EPCN 1711188, AMPS 1736448 and CAREER 1752362, and Johns Hopkins University Discovery Award.}
\thanks{P. You, D. F. Gayme and E. Mallada are with the Whiting School of Engineering, Johns Hopkins University, Baltimore, MD 21218, USA
        {\tt\small \{pcyou,mallada,dennice\}@jhu.edu}}%
}

\begin{document}

\maketitle
\thispagestyle{empty}
\pagestyle{empty}

%%%%%%%%%%%%%%%%%%%%%%%%%%%%%%%%%%%%%%%%%%%%%%%%%%%%%%%%%%%%%%%%%%%%%%%%%%%%%%%%
\begin{abstract}

Two-stage electricity market clearing is designed to maintain market efficiency under ideal conditions, e.g., perfect forecast and nonstrategic generation.
This work demonstrates that the individual strategic behavior of inelastic load participants in a two-stage settlement electricity market can deteriorate efficiency. 
Our analysis further implies that virtual bidding can play a role in alleviating this loss of efficiency by mitigating the market power of strategic load participants.
We use real-world market data from New York ISO to validate our theory.

\end{abstract}

%%%%%%%%%%%%%%%%%%%%%%%%%%%%%%%%%%%%%%%%%%%%%%%%%%%%%%%%%%%%%%%%%%%%%%%%%%%%%%%%

%\begin{table}
%\caption{An Example of a Table}
%\label{table_example}
%\begin{center}
%\begin{tabular}{|c||c|}
%\hline
%One & Two\\
%\hline
%Three & Four\\
%\hline
%\end{tabular}
%\end{center}
%\end{table}

%%%%%%%%%%%%%%%%%%%%%%%%%%%%%%%%%%%%%%%%%%%%%%%%%%%%%%%%%%%%%%%%%%%%%%%%%%%%%%%%

%\addtolength{\textheight}{-3cm}   % This command serves to balance the column lengths

%%%%%%%%%%%%%%%%%%%%%%%%%%%%%%%%%%%%%%%%%%%%%%%%%%%%%%%%%%%%%%%%%%%%%%%%%%%%%%%%

%   \begin{figure}[thpb]
%      \centering
%      %\includegraphics[scale=1.0]{figurefile}
%      \caption{Inductance of oscillation winding on amorphous
%       magnetic core versus DC bias magnetic field}
%      \label{figurelabel}
%   \end{figure}

%%%%%%%%%%%%%%%%%%%%%%%%%%%%%%%%%%%%%%%%%%%%%%%%%%%%%%%%%%%%%%%%%%%%%%%%%%%%%%%%
%%%%%%%%%%%%%%%%%%%%%%%%%%%%%%%%%%%%%%%%%%%%%%%%%%%%%%%%%%%%%%%%%%%%%%%%%%%%%%%%
%%%%%%%%%%%%%%%%%%%%%%%%%%%%%%%%%%%%%%%%%%%%%%%%%%%%%%%%%%%%%%%%%%%%%%%%%%%%%%%%

\section{Introduction}

Electricity markets are designed to complement physical power systems by utilizing prices or other monetary incentives to motivate efficient system operation.
Wholesale electricity markets generally consist of two-stage settlement.
The first stage is a day-ahead market where participants buy or sell electricity through bids or offers on an hourly basis. An independent system operator (ISO) determines the hourly generation and load schedules along with the corresponding day-ahead clearing prices for the next day.
The second stage is a real-time market where participants trade in the same way at the real-time clearing prices on a smaller timescale, usually every five minutes, to offset any discrepancy between day-ahead commitments and actual generation/load.

The day-ahead and real-time markets are tightly coupled via time-varying supply, demand and prices \cite{pang2018temporally}.
The two-stage settlement is designed to maintain equal day-ahead and real-time prices such that no speculator is able to perform arbitrage, i.e., to enforce the so-called \emph{no-arbitrage} condition. 
However, the two stages are settled separately in practice and identical prices in the day-ahead and real-time markets are therefore not directly enforced \cite{NYISOprice}. 
The difference between a day-ahead price and its real-time counterpart is technically termed a price \emph{spread}. 
Any nonzero spread is generally considered a loss of efficiency \cite{hogan2016virtual}. 
Situations that result in systematic nonzero spreads include external factors, such as load forecast errors \cite{tang2016model}, non-scheduled generator shutdowns or line maintenance, as well as internal market power generally exercised by strategic generators \cite{ruhi2018opportunities}.

Transactions that are not intended for physical fulfillment in real time but holding financial positions for arbitrage are referred to as \emph{virtual bids}.
Virtual bids primarily consist of \emph{decrement} bids that buy electricity in the day-ahead market with the obligation to sell back the same amount in the real-time market, as well as \emph{increment} offers that work exactly in the opposite way \cite{PJMvirtualbidding}. See \cite{baltaoglu2019algorithmic,mashhour2011bidding1,mashhour2011bidding2,rahimiyan2016strategic} for various examples of virtual bidding strategies.
Virtual bidding is a valuable component of the two-stage settlement design that contributes to increasing market liquidity and mitigating market power by allowing extra asset-free participants to compete in electricity markets. This practice has proven, through both real observation \cite{hadsell2007impact} and theoretical analysis \cite{hogan2016virtual,isemonger2006benefits,celebi2010virtual,mather2017virtual}, to improve market efficiency by driving day-ahead and real-time prices to converge.
%There is still a lot of debate on the performance of virtual bidding.
%Although dissenters are concerned about the lack of a general framework for evaluating the overall benefit and cost of virtual bidding, generally theoretical analysis of its mechanics and illustrative examples from the industry are supportive of its positive effect on market efficiency.
%The New York ISO wholesale power market has observed reduced volatility in both day-ahead and real-time prices from real-world data since virtual bidding was introduced in 2001 \cite{hadsell2007impact}.
%In \cite{mather2017virtual} the role of virtual bidding in mitigating the excess cost incurred by myopic scheduling on the ISO side is explored. Through a game-theoretical model, it shows that virtual bidding contributes to the socially optimal day-ahead schedule at equilibrium.

Despite the aforementioned studies, little attention has yet been paid to the strategic behavior of load participants in electricity markets, which may also play a role in degrading market efficiency. The load side is usually less regulated due to its inelasticity, which leaves load participants more freedom to make strategic decisions. 
Conceptually, even with inelastic demand, a load participant enjoys the flexibility of two-stage settlement, which potentially enables it to exercise market power.

%The discrepancy between day-ahead and real-time prices is officially termed as the \emph{spread}.

%Moreover, in practice nonzero spreads arise from various external factors, e.g., load forecast errors, non-scheduled generator shutdowns or line maintenance.
%Of particular relevance to this work is a data-driven approach developed in \cite{tang2016model} that establishes a game-theoretical setting in a two-stage settlement market model with virtual bidders as strategic players. The discrepancy between day-ahead and real-time prices is interpreted as a measure of the average forecast accuracy of the market and all virtual bidders, which converges to zero with more and more qualified virtual bidders with high forecast precision involved.

%Nonetheless, we suspect that there exists extra internal factors that lead to systematic nonzero spreads and inefficiency, e.g., strategic behavior of market participants.

In this paper, we look at the role of strategic inelastic load participants that take advantage of the two-stage settlement mechanism.
We first establish a simple two-stage settlement market model that assumes (fully regulated) nonstrategic generation to characterize the inherent connection between the day-ahead and real-time markets. 
The strategic behavior of load participants is then analyzed through a Cournot game. 
We further extend the framework to accommodate decrement bids in virtual bidding as a special case of strategic inelastic load participation in electricity markets. 
Real-world market data from New York ISO (NYISO) are employed for validation.

Our analysis has multiple implications.
\emph{First}, the proposed market model unveils the underlying mechanism that relates the no-arbitrage condition with market efficiency while maintaining realistic market settlement conditions such as the day-ahead cleared load being approximately equal to the total load for efficiency.  
\emph{Second}, we identify adverse impacts of strategic behavior by inelastic load participants that induces negative spreads and deteriorates efficiency in electricity markets, despite perfect forecast and nonstrategic generation.
\emph{Third}, we show that virtual bidding is an effective solution to alleviating the loss of market efficiency caused by strategic load participants.

The rest of the paper is organized as follows. Section~\ref{sec:model} introduces our electricity market model.
The role of strategic behavior by inelastic load participants is then analyzed in Section~\ref{sec:strategicload}. Empirical validation using real-world data follows in Section~\ref{sec:validation}. Section~\ref{sec:conclusion} concludes the paper.

%The no-arbitrage condition that enforces a theoretical zero spread indicates the stability of a financial equilibrium. However, the explicit correlation between the no-arbitrage condition and market efficiency has rarely been demonstrated. For instance, the two-stage settlement models in \cite{bessembinder2002equilibrium,zhang2009analyzing} fail to account for the fact that the day-ahead cleared load should approximately amount to the total load for efficiency.

%\enrique{As of now the introduction is too long and somehow cumbersome. I suggest the following paragraphs (some/most of the overalap with what you have already):}
%{\color{bleudefrance}
%\begin{enumerate}
%\item Electricity markets broad description (DA/RT)
%\item Empirical evidence show lack of efficiency (attributed to forecast errors or generators  market power),
%\item Virtual bidders diminishing spread..
%% \item Virtual bidders Role
%\item Little attention to strategic load behavior
%\item In this paper, we look at the role of....
%(This paper contributions).
%\item Highlight of findings ()
%\item Organization of the rest of the paper.
%\end{enumerate}
%}

%%%%%%%%%%%%%%%%%%%%%%%%%%%%%%%%%%%%%%%%%%%%%%%%%%%%%%%%%%%%%%%%%%%%%%%%%%%%%%%%
%%%%%%%%%%%%%%%%%%%%%%%%%%%%%%%%%%%%%%%%%%%%%%%%%%%%%%%%%%%%%%%%%%%%%%%%%%%%%%%%
%%%%%%%%%%%%%%%%%%%%%%%%%%%%%%%%%%%%%%%%%%%%%%%%%%%%%%%%%%%%%%%%%%%%%%%%%%%%%%%%

\section{Electricity Markets Model}\label{sec:model}

In this section we describe the proposed electricity market model for the two-stage settlement mechanism. 
Consider an electricity market that consists of a day-ahead market and a real-time market.
Assume that the generation side is highly regulated and all generators are non-strategic, i.e., they reveal their true cost functions\footnote{In real electricity markets, piecewise linear generation offers are made as a proxy for true generation cost functions, which are assumed to be known by the ISO here for ease of analysis.}. The generators are categorized into two sets based on whether they are sufficiently fast to participate in the real-time market.
Let $\mathcal{F}$ and $\mathcal{S}$ respectively denote the sets of fast-responsive and slow-responsive generators. Slow-responsive generators can only participate in the day-ahead market while fast-responsive generators are able to participate in both markets.
For a fast-responsive generator $i\in\mathcal{F}$ that outputs an amount of power $x^f_i\ge0$, we assume a quadratic cost function of the form
\beq\label{eq:fastcostfunction}
C^f_i(x^f_i):=\frac{\alpha^f_{i}}{2}x^{f2}_i + \beta^f_i x^f_i ,
\eeq
where $\alpha^f_i>0$ and $\beta^f_i$ are constant cost coefficients. Similarly, we denote the cost function of a slow-responsive generator $j\in\mathcal{S}$ by
\beq\label{eq:slowcostfunction}
C^s_j(x^s_j):=\frac{\alpha^s_{j}}{2}x^{s2}_j + \beta^s_j x^s_j ,
\eeq
where $x^s_j\ge 0 $, $\alpha^s_j>0$ and $\beta^s_j$ are defined accordingly.
%We enforce $\alpha^f_i >\alpha^s_j$ for any $i\in\mathcal{F}, j\in\mathcal{S}$.
We then define the associated vectors as $x^f := (x^f_i,i\in\mathcal{F})$ and $x^s := (x^s_j,j\in\mathcal{S})$.

%\enrique{Not sure about this section organization, what appears below seems to be about efficiency. Is this part of Supply Elasticity?}

\subsection{Two-Stage Settlement}

%\enrique{The ``Two-Stage Settlement'' subsection should come here, after the model is introduced.}
The two-stage settlement mechanism meets a total inelastic load of $d>0$ by clearing it efficiently but separately in the day-ahead and real-time markets.
Denote the day-ahead cleared portion as $d^{DA}$ and the real-time cleared portion as $d^{RT}$ such that $d=d^{DA}+d^{RT}$.
In the slow-timescale day-ahead market, all of the generators in $\mathcal{F}$ and $\mathcal{S}$ are involved to clear the load $d^{DA}$ based on the following:

\noindent
\textbf{Day-ahead market clearing problem}
\begin{subequations}\label{eq:DayAheadMin}
\bq\label{eq:DayAheadMin.a}
\min_{x^f,x^s\ge 0} &&  \sum_{i\in\mathcal{F}} C^f_i(x^f_i) + \sum_{j\in\mathcal{S}} C^s_j(x^s_j)  \\
\label{eq:DayAheadMin.b}
\mathrm{s.t.}  && \sum_{i\in\mathcal{F}}x^f_i + \sum_{j\in\mathcal{S}} x^s_j = d^{DA}  \ : \ \lambda^{DA},
\eq
\end{subequations}
where $\lambda^{DA}$ is the dual Lagrange multiplier for the equality constraint \eqref{eq:DayAheadMin.b}.
%The only difference here compared with \eqref{eq:SocialCostMin} is that the day-ahead cleared load is not necessarily the total load. 
Due to strong convexity, \eqref{eq:DayAheadMin} has a unique minimizer which we denote as $(x^{f*},x^{s*})$.      
Since \eqref{eq:DayAheadMin.b} is affine, the KKT conditions suggest that all of the participating generators should have an identical marginal cost that equals the optimal dual Lagrange multiplier\footnote{For illustration purposes, throughout this paper we restrict our considerations to the case where the constraints $x_i^f\ge 0, i\in\mathcal{F}$ and $x_j^s \ge 0,  j\in\mathcal{S}$ are not binding.}:
\bq\label{eq:DayAheadMarginalCost}
\lambda^{DA}= \alpha^f_i x^{f*}_i + \beta^f_i =\alpha^s_j x^{s*}_j +\beta^s_j, \quad \forall i \in\mathcal{F}, j\in\mathcal{S},
\eq
where we abuse $\lambda^{DA}$ to denote its optimum. $\lambda^{DA}$, technically termed the \emph{shadow price} in economics \cite{kirschen2018fundamentals}, is the minimum price to incentivize generators to output the desired amount of power, which captures marginal generation cost.

Combining \eqref{eq:DayAheadMin.b} and \eqref{eq:DayAheadMarginalCost} results in
\begin{subequations}\label{eq:DAprice}
\beq
 \lambda^{DA}  =   \alpha^{DA} d^{DA}  +    \beta^{DA} ,
\eeq
where 
\beq
\alpha^{DA} := \left(\sum_{i\in\mathcal{F}} \frac{1}{\alpha^f_i}   +     \sum_{j\in\mathcal{S}} \frac{1}{\alpha^s_j}  \right)^{-1}  ,
\eeq
and 
\beq
\beta^{DA} \! := \!  \left(\! \sum_{i\in\mathcal{F}} \frac{1}{\alpha^f_i}   +     \sum_{j\in\mathcal{S}} \frac{1}{\alpha^s_j} \! \right)^{-1} \!   \left(\! \sum_{i\in\mathcal{F}} \frac{\beta^f_i}{\alpha^f_i}  +\sum_{j\in\mathcal{S}} \frac{\beta^s_j}{\alpha^s_j}   \! \right)  .
\eeq
\end{subequations}
Here $\alpha^{DA}$ and $\beta^{DA}$ serve as the aggregate pricing coefficients.
The expressions in \eqref{eq:DAprice} implicitly reflect the \emph{elasticity of supply}, defined as the responsiveness of the quantity of power supplied to a change in its price, in the day-ahead market. Basically, given the market price $\lambda^{DA}$, the generators are willing to output a total amount of power $d^{DA}$. In other words, to clear the load $d^{DA}$ in the day-ahead market, the clearing price needs to be set at $\lambda^{DA}$.

The fast-timescale real-time market clears in the same way as the day-ahead market except that only fast-responsive generators in $\mathcal{F}$ are involved. Note that these generators have also participated in the day-ahead market and have already been scheduled to output $x^{f*}$. Therefore, in order to clear the load $d^{RT}$, the real-time market solves the following optimization problem: 

\noindent
\textbf{Real-time market clearing problem}
\begin{subequations}\label{eq:RTmin}
\bq\label{eq:RTmin.a}
\min_{\delta x^f} &&  \sum_{i\in\mathcal{F}} C^f_i( x^{f*}_i + \delta x^f_i )    \\\label{eq:RTmin.b}
\mathrm{s.t.}  && \sum_{i\in\mathcal{F}} \delta x^f_i = d^{RT}  \ : \  \lambda^{RT} .
\eq
\end{subequations}
Here $\delta x^f_i $ denotes the scheduled output adjustment from $x^{f*}$ for generator $i\in\mathcal{F}$ and $\delta x^f:= (\delta x^f_i,i\in\mathcal{F})$. $\lambda^{RT}$ is the (optimal) dual Lagrange multiplier for the equality constraint \eqref{eq:RTmin.b}.
Note that the cost of fast-responsive generators in the day-ahead market, i.e., $\sum_{i\in\mathcal{F}} ( \frac{\alpha^f_i}{2} x^{f*2}_i+ \beta^f_i x^{f*}_i  )$, should be subtracted from the objective function to represent the exact total cost for clearing the real-time load $d^{RT}$. We ignore this constant term for brevity.

We denote the unique minimizer of \eqref{eq:RTmin} as $\delta x^{f*}$ and deduce the following from the KKT conditions:
\beq\label{eq:RTmarginalcost}
\lambda^{RT}  = \alpha^f_i (x^{f*}_i  +  \delta x^{f*}_i  ) +  \beta^f_i  
 =\alpha^f_i  \delta x^{f*}_i  + \lambda^{DA}, \quad \forall i \in\mathcal{F},
\eeq
where the second equality follows directly from \eqref{eq:DayAheadMarginalCost}.
Substituting \eqref{eq:RTmarginalcost} into \eqref{eq:RTmin.b} yields
\begin{subequations}\label{eq:RTprice}
\beq
\lambda^{RT} =  \alpha^{RT} d^{RT} + \beta^{RT},
\eeq
where
\beq
\alpha^{RT} :=   \left(\sum_{i\in\mathcal{F}} \frac{1}{\alpha^f_i}   \right)^{-1} 
\eeq
and 
\beq
\beta^{RT} :=  \lambda^{DA}.
\eeq
\end{subequations}
Here $\alpha^{RT}$ and $\beta^{RT}$ are the aggregate pricing coefficients that embody the elasticity of supply in the real-time market. Meanwhile, \eqref{eq:RTprice} also unveils the inherent correlation between the day-ahead and real-time prices: the latter deviates from the former to account for the real-time cleared load. See Fig.~\ref{fig:pricecorrelation}. Formally, the price spread between the day-ahead and real-time prices is 
\bq\label{eq:spread}
 \lambda^{DA}- \lambda^{RT} =  - \left(\sum_{i\in\mathcal{F}} \frac{1}{\alpha^f_i}   \right)^{-1} d^{RT} .
\eq

\begin{figure}[thpb]
	\centering
	\includegraphics[width=0.45\textwidth]{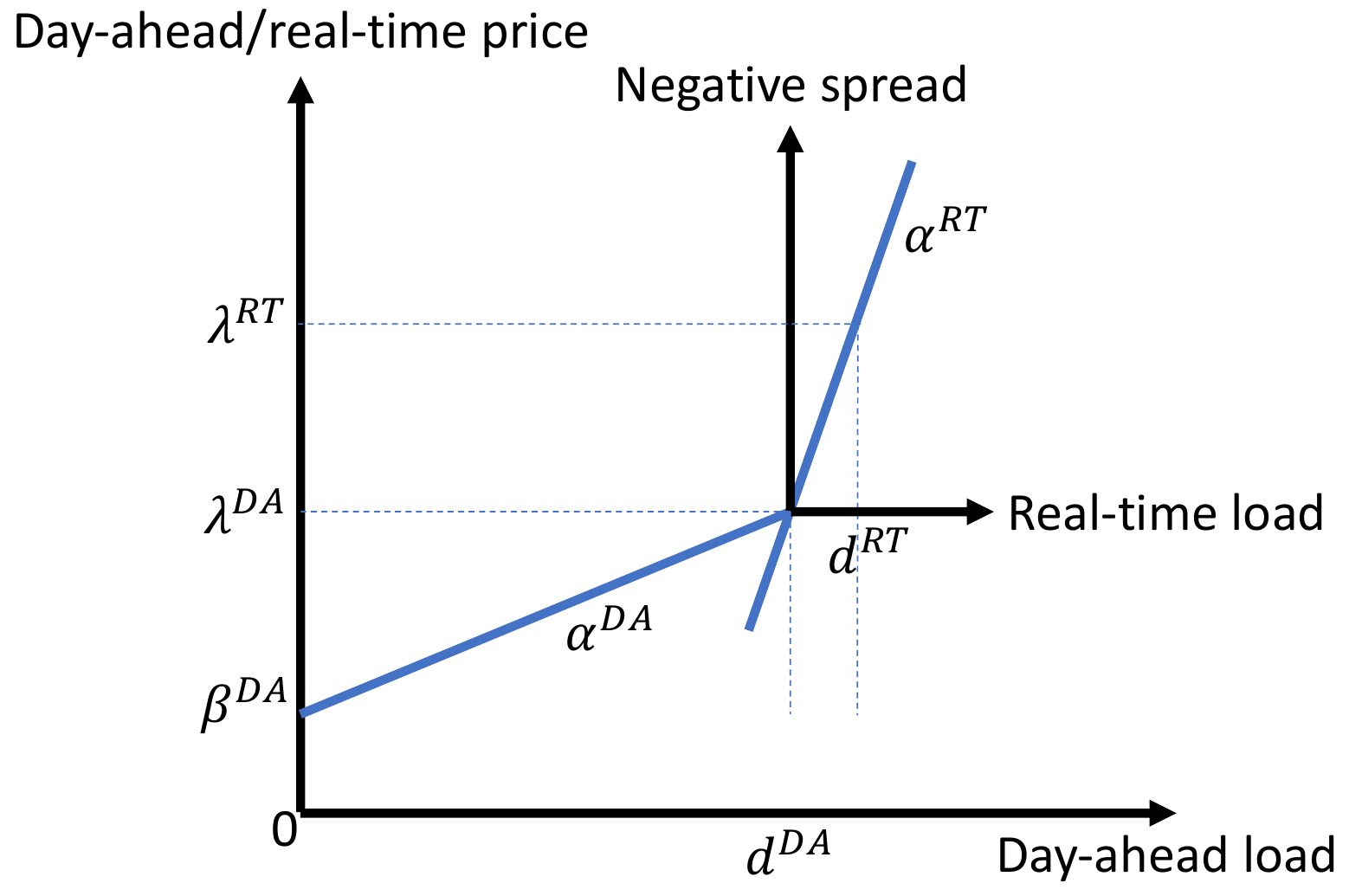}
	\caption{Correlation between day-ahead and real-time prices.}
	\label{fig:pricecorrelation}
\end{figure}

\begin{remark}
			Notably, according to \eqref{eq:DAprice}, \eqref{eq:RTprice}, we always have $\alpha^{RT} > \alpha^{DA}>0$, as Fig.~\ref{fig:pricecorrelation} illustrates, due to a smaller subset of generators involved in the real-time market. This is consistent with the observation that real-time prices are more volatile than day-ahead prices, since the real-time market typically has a smaller price elasticity of supply than the day-ahead market, i.e., the quantity of power supply in the real-time market is less sensitive to a change in its price than that in the day-ahead market.
\end{remark}

\subsection{Market Efficiency}

We next formalize our definition of market efficiency. Given all the available generators in $\mathcal{F}$ and $\mathcal{S}$, we define market efficiency as the minimum of social cost to meet the total inelastic load $d$, which is specifically realized by solving the following:

\noindent
\textbf{Social cost minimization problem}
\begin{subequations}\label{eq:SocialCostMin}
	\bq\label{eq:SocialCostMin.a}
	\min_{x^f,x^s\ge 0} &&  \sum_{i\in\mathcal{F}} C^f_i(x^f_i) + \sum_{j\in\mathcal{S}} C^s_j(x^s_j) \\ \label{eq:SocialCostMin.b}
	\mathrm{s.t.}  && \sum_{i\in\mathcal{F}}x^f_i + \sum_{j\in\mathcal{S}} x^s_j = d  \ : \ \lambda,
	\eq
\end{subequations}
i.e., jointly optimizing the dispatch of all the generators across the two markets. 
We define the (optimal) dual Lagrange multiplier $\lambda$ for the equality constraint \eqref{eq:SocialCostMin.b} and denote the unique minimizer of \eqref{eq:SocialCostMin} by $(x^{f\#},x^{s\#})$.    
The KKT conditions require
\bq\label{eq:efficiency}
\lambda = \alpha^{f}_i x^{f\#}_i + \beta^f_i  = \alpha^s_j x^{s\#}_j + \beta^s_j, \quad \forall i\in\mathcal{F},j\in\mathcal{S},
\eq
i.e., equal marginal cost for all of the participating generators, to achieve efficiency.

Recall that the day-ahead price equals the marginal cost of slow-responsive generators in \eqref{eq:DayAheadMarginalCost} while the real-time price equals the marginal cost of fast-responsive generators in \eqref{eq:RTmarginalcost}. By comparing \eqref{eq:DayAheadMarginalCost} and \eqref{eq:RTmarginalcost} with the indicator of market efficiency \eqref{eq:efficiency}, we arrive at the following theorem:
\begin{theorem}\label{teo:efficiency}
	In the two-stage settlement electricity market, efficiency can only be realized when 
	\beq\label{eq:indication1_efficiency}
	\lambda^{DA}= \lambda^{RT} = \lambda
	\eeq
	i.e., the day-ahead and real-time prices equalize, 
	which further implies
	\beq\label{eq:indication2_efficiency}
	 d^{DA} = d, \quad d^{RT} = 0 .
	\eeq
%	, i.e., $x_i^{f*} = x_i^{f\#}$, $\forall i\in\mathcal{F}$ and $x_j^{s*} = x_j^{s\#}$, $\forall j\in\mathcal{S}$.
\end{theorem}
	%Any nonzero spread $ - \left(\sum_{i\in\mathcal{F}} \frac{1}{\alpha^f_i}   \right)^{-1} d^{RT} \neq 0$ between $\lambda^{DA}$ and $\lambda^{RT}$ results in inefficiency. In order to eliminate the discrepancy and achieve a zero spread, $d^{DA} = d$ is required from \eqref{eq:RTprice}. 
	Theorem \ref{teo:efficiency} matches exactly the intuition of the two-stage settlement design: all (forecast) load should be cleared in the day-ahead market while the real-time market accounts for any load deviation from the forecast. 
	It also suggests that efficiency is consistent with the no-arbitrage condition between the two-stage markets, guaranteed by the zero spread from \eqref{eq:indication1_efficiency}, which is necessary for the market model to be realistic.
	%which implies that any nonzero real-time load $d^{RT}$ that the market dispatches is a sign of uncertainty. 
	
Similar models for the two-stage settlement mechanism have been used in \cite{bessembinder2002equilibrium,zhang2009analyzing,tang2016model}. 
However, our simple model further addresses several issues that are missing in these previous works,
e.g., the fact that the day-ahead cleared load should equal the total load is not accounted for in \cite{bessembinder2002equilibrium,zhang2009analyzing}; the correlation between the no-arbitrage condition and market efficiency is not demonstrated in \cite{tang2016model}.

%\enrique{At this point I'm wondering if we want to say anything about how to motivate this model. That is, ``similar models have been used in XXX, the main difference is...''}

%%%%%%%%%%%%%%%%%%%%%%%%%%%%%%%%%%%%%%%%%%%%%%%%%%%%%%%%%%%%%%%%%%%%%%%%%%%%%%%%
%%%%%%%%%%%%%%%%%%%%%%%%%%%%%%%%%%%%%%%%%%%%%%%%%%%%%%%%%%%%%%%%%%%%%%%%%%%%%%%%
%%%%%%%%%%%%%%%%%%%%%%%%%%%%%%%%%%%%%%%%%%%%%%%%%%%%%%%%%%%%%%%%%%%%%%%%%%%%%%%%

\section{Strategic Load Participant}\label{sec:strategicload}

Given the two-stage settlement mechanism, an electricity market should clear all of the load in the day-ahead market and zero load in the real-time market in order to achieve efficiency. However, we observe that in the NYISO market there is an obvious positive bias for real-time loads throughout the year of 2018, as shown in Fig~\ref{fig:loadprofile}, which cannot be accounted for by uncertainties.
We attribute this loss of efficiency to strategic behavior by inelastic load participants and next investigate their market power by taking advantage of the two-stage settlement mechanism. Ideal assumptions of perfect forecast and nonstrategic generation are made to focus our attention on the impact of strategic load.
As we will see below, our analysis extends naturally to accommodate the role of virtual bidding.

\begin{figure}[thpb]
	\centering
	\includegraphics[width=0.45\textwidth]{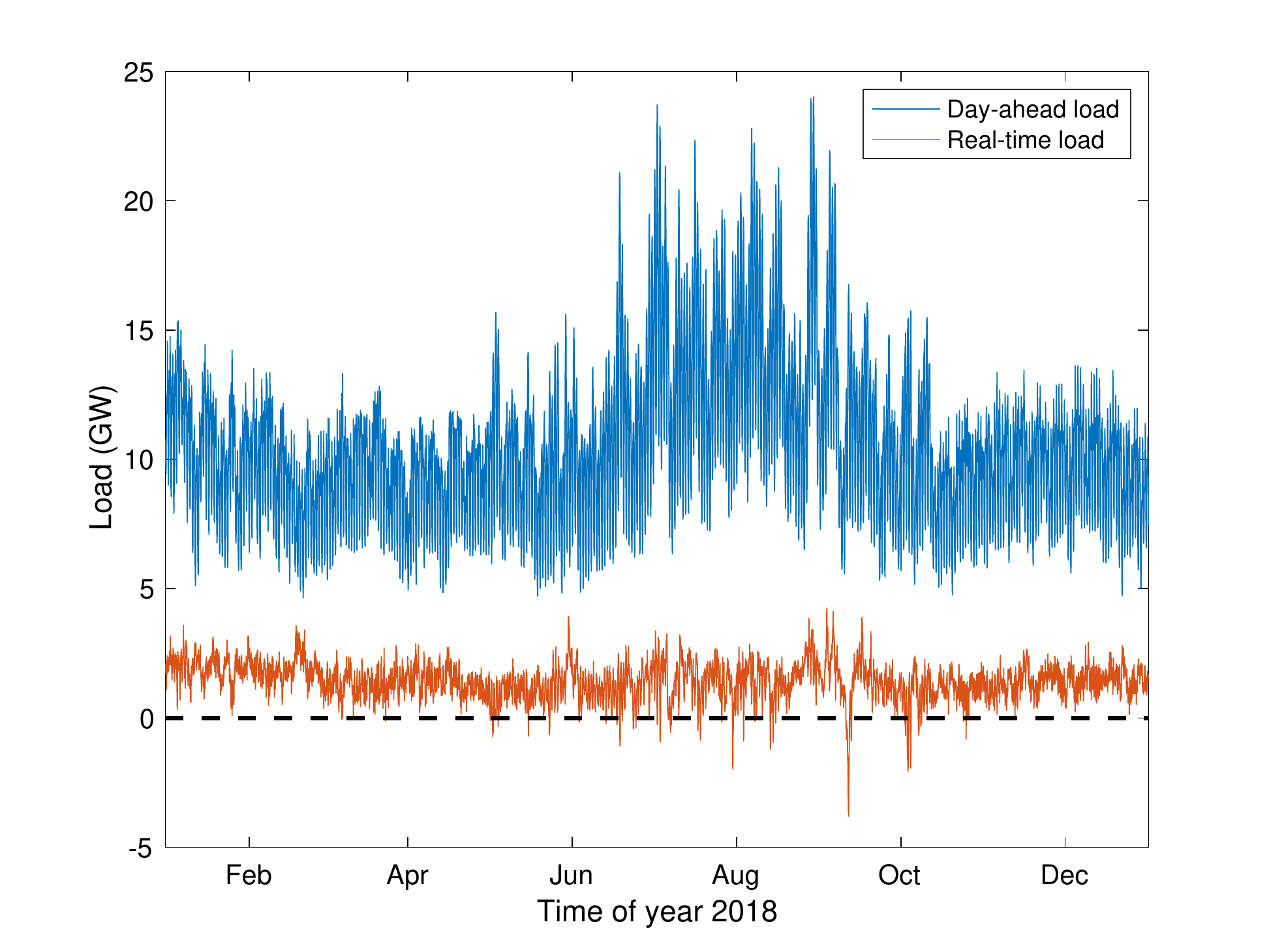}
	\caption{Day-ahead and real-time cleared loads, NYISO, 2018.}
	\label{fig:loadprofile}
\end{figure}

\subsection{Single Load}

We start with the simplest case where there is only one single inelastic load $d$ to be cleared. It has the option to participate in either one of or both of the day-ahead and real-time markets to meet its demand. The participation of the load in the two markets affects the market clearing prices, which in turn determine its cost. For analysis purposes, we assume the load has full knowledge of the supply elasticity of both markets, i.e., it knows the exact values of $\alpha^{DA},\beta^{DA},\alpha^{RT},\beta^{DA}$, e.g., through estimates based on long-term historical data\footnote{Since the set of participating generators in an electricity market and their cost functions are usually stable subject to subtle changes in the long run, it is reasonable to argue those coefficients that characterize the dependence of market prices on the amount of cleared load are approximately constant and easy to estimate.}.

A strategic load will anticipate the impact of its decision on the two markets and minimize the expenditure of purchasing electricity to meet its demand accordingly. Formally, it solves the following:

\noindent
\textbf{Single load cost minimization problem}\footnote{Note that on the load side we constrain nonnegative load participation in the day-ahead market to maintain the identity as a load.}
%\emumargin{Maybe use $\lambda^{DA}(d^{DA})$ here?}{}
\begin{subequations}\label{eq:singleloadcostmin}
\bq\label{eq:singleloadcostmin.a}
\!\!\!\!\!\!\!\!\!\!\!\!\!   \min_{d^{DA}\ge 0,d^{RT}} && \!\!\!\!\!\! \!\!\!  \lambda^{DA}(d^{DA}) \cdot d^{DA} + \lambda^{RT}(d^{DA},d^{RT}) \cdot d^{RT} \\\label{eq:singleloadcostmin.b}
\mathrm{s.t.}  \  \  && \!\!\!\!\!\!  \!\!\! d^{DA} + d^{RT} =d .
\eq
\end{subequations}

\begin{theorem}
	The optimal load participation for a single load in the two-stage settlement electricity market is uniquely determined by
	\bq\label{eq:singloadloadminimizer}
	d^{DA*}=\left(1-\frac{\alpha^{DA}}{2\alpha^{RT}} \right) d, \quad d^{RT*}=\frac{\alpha^{DA}}{2\alpha^{RT}} \ d  ,
	\eq
	i.e., $d>d^{DA*}>\frac{1}{2}d>0$ and $\frac{1}{2}d>d^{RT*}>0$.
	Therefore, $\lambda^{RT}> \lambda^{DA}$ and a strictly negative spread, as defined in \eqref{eq:spread}, follow.
\end{theorem}
\begin{proof}
	First of all, we relax the constraint $d^{DA}\ge 0$ and show that it is not binding at the optimum.
	Given the explicit expressions of $\lambda^{DA}(d^{DA})$ and $\lambda^{RT}(d^{DA},d^{RT})$ in \eqref{eq:DAprice} and \eqref{eq:RTprice}, we can substitute \eqref{eq:singleloadcostmin.b} into \eqref{eq:singleloadcostmin.a} to reorganize the objective function in terms of $d^{DA}$ only:
	\beq\nonumber
	\begin{aligned}
		\ & \lambda^{DA}(d^{DA})\cdot d^{DA} + \lambda^{RT}(d^{DA},d^{RT}) \cdot d^{RT} \\
		=\ & \lambda^{DA}(d^{DA})\cdot d^{DA} + (\alpha^{RT} d^{RT} +  \lambda^{DA}(d^{DA})) \cdot d^{RT} \\
		=\ & (\alpha^{DA}d^{DA}+\beta^{DA}) d + \alpha^{RT} (d-d^{DA})^2 \\
		=\ & \alpha^{RT} d^{DA2} + (\alpha^{DA}d  -  2\alpha^{RT}d) d^{DA} + \beta^{DA}d + \alpha^{RT}d^2 . 
	\end{aligned}
	\eeq
	The unique minimizer of the above unconstrained optimization is straightforwardly obtained by the first-order optimality condition, i.e., \eqref{eq:singloadloadminimizer}. Recall $\alpha^{RT} > \alpha^{DA} > 0$, which implies $d>d^{DA*}>\frac{1}{2}d>0$ and $\frac{1}{2}d>d^{RT*}>0$. 
	The relaxed constraint is satisfied and \eqref{eq:singloadloadminimizer} is also the unique optimum of \eqref{eq:singleloadcostmin}.
\end{proof}

\begin{remark}
%	When there is only one strategic load participant in a two-stage settlement electricity market, $\lambda^{RT}>\lambda^{DA}$ and 
	The negative spread indicates the loss of market efficiency caused by the strategic behavior of a single inelastic load participant in the two-stage settlement electricity market. Meanwhile, the strictly positive load participation in the real-time market coincides with the observation of positive bias for real-time loads in Fig.~\ref{fig:loadprofile}.
\end{remark}

%\enrique{This paragraph, and several other on the paper, have many equations and derivations in within. As a general rule, if the equation is too tall to normally fit in a line, make it big with equation or align.}

The single-load case serves as a toy example. %to show that the strategic behavior of load participants even with inelastic demand hurts market efficiency.
%However, the market paradigm for efficiency has been broadly recognized and practiced in the power industry worldwide.
Next we proceed to characterize the general case with market competition and analyze its impact on efficiency.

\subsection{Load-Side Cournot Competition}

We extend the above analysis to the case with multiple individual strategic loads, e.g., different local utility companies in a market.
Let $\mathcal{L}:=\{1,2,\dots,L\}$ be the set of these loads. Each load $l\in\mathcal{L}$ can independently determine its participation, $\vec d_l:= (d^{DA}_l\ge0,d^{RT}_l)$, in the day-ahead and real-time markets in order to satisfy its inelastic demand $d_l$ with 
\beq\label{eq:gameloadconstraint}
d^{DA}_l+d^{RT}_l=d_l,\quad l\in\mathcal{L} . 
\eeq
Let $\vec d := (\vec d_l, l\in \mathcal{L})$ be the aggregate decisions for all of the loads. Further denote the aggregate decisions for all of the loads except load $l$ as $\vec d_{-l}$.
Suppose that all of the loads are aware of the mechanism that determines market prices, i.e., 
\beq\label{eq:gamepricing}
\lambda^{DA}=\alpha^{DA} d^{DA} + \beta^{DA}, \quad \lambda^{RT}=\alpha^{RT} d^{RT} + \beta^{RT} , 
\eeq
where $d^{DA} := \sum_{l\in\mathcal{L}} d_l^{DA}$ and $d^{RT}:= \sum_{l\in\mathcal{L}} d_l^{RT}$. 
Define $d:= \sum_{l\in\mathcal{L}} d_l$ as the total load to be cleared.
Each load $l$ will aim to minimize its expenditure of purchasing electricity from the two markets to meet demand given other loads' decisions, i.e., 
\begin{subequations}\label{eq:gameloadexp}
\bq
 \!\!\!\!\!\!  \!\!\!\!\!\!   \min_{\vec{d}_l}  &&  \!\!\!\!\!\! c_l(\vec d_l;\vec{d}_{-l}): =\lambda^{DA}(\vec{d}) \cdot d^{DA}_l  + \lambda^{RT}(\vec{d}) \cdot d^{RT}_l   \\
 \!\!\!\!\!\!  \!\!\!\!\!\!  \mathrm{s.t.} && \!\!\!\!\!\!  \eqref{eq:gameloadconstraint} .
\eq
\end{subequations}
%\enrique{Again, here too many definitions to be used later that are defined within a paragraph and cannot be referenced.}

These loads compete in quantities of participation in the two markets that affect market clearing prices and seek to minimize individual cost, which can be formalized as a Cournot game:

\noindent
\textbf{Load-side Cournot game}\\
\emph{Players}: each load $l\in\mathcal{L}$;\\
\emph{Strategies}: load participation $\vec d_l$ in the day-ahead and real-time markets to satisfy \eqref{eq:gameloadconstraint};  \\
\emph{Costs}: expenditure of purchasing electricity $c_l(\vec d_l;\vec{d}_{-l})$.\\

\begin{definition}
$\vec d^*$ is a Nash equilibrium of the load-side Cournot game if it satisfies $c_l(\vec d^*) \le c_l(\vec d_l; \vec d^*_{-l})$ for any $\vec d_l$, $\forall l\in\mathcal{L}$.
\end{definition}
At a Nash equilibrium, no load has the incentive to deviate from its current decision unilaterally, given other loads' decisions. In order to characterize the Nash equilibrium of the load-side Cournot game, we first propose the following lemma:
\begin{lemma}\label{lemma:gameeq}
	There do not exist equilibria of the load-side Cournot game where $d_l^{DA*}=0$ for some $l\in\mathcal{L}$.
\end{lemma}
Refer to the appendix for the proof.
Given Lemma \ref{lemma:gameeq}, the possibility of Nash equilibria with any of the constraints $d^{DA}_l \ge 0,~l\in\mathcal{L}$ binding is excluded and we next prove the existence and uniqueness of the Nash equilibrium of the load-side Cournot game by ignoring these constraints:
\begin{theorem}\label{prop:nasheq}
	In the two-stage settlement electricity market,
	there exists a unique Nash equilibrium of the load-side Cournot game, characterized by 
	\begin{small}
	\beq\label{eq:gameeq3}
	\begin{aligned}
		&\!\!\! d_l^{DA*} = \left(1 -\frac{L \alpha^{DA}}{(L+1)\alpha^{RT}}\right) d_l  +   \frac{\alpha^{DA}}{(L+1)\alpha^{RT}}  \sum_{k\in\mathcal{L}\backslash \{l\}} d_k , \\
		&\!\!\! d_l^{RT*} = \frac{L \alpha^{DA}}{(L+1)\alpha^{RT}} d_l - \frac{\alpha^{DA}}{(L+1)\alpha^{RT}}  \sum_{k\in\mathcal{L}\backslash \{l\}} d_k,
	\end{aligned}
	\eeq	\end{small}%
	for $\forall l \in\mathcal{L}$.
\end{theorem}
\begin{proof}
Given \eqref{eq:gameloadconstraint} and \eqref{eq:gamepricing}, the expenditure function $c_l(\vec d_l;\vec{d}_{-l})$ of each load $l$ in \eqref{eq:gameloadexp} can be rewritten explicitly in terms of $d^{DA}_l$ only as follows.
\begin{small}
%\begin{subequations}\label{eq:expenditurefunc}
\bq\nonumber
	 &&\!\!\! \lambda^{DA} d^{DA}_l  + \lambda^{RT} d^{RT}_l \\\nonumber
	&= & \!\!\! ( \alpha^{DA} \sum_{k\in\mathcal{L}} d_k^{DA} + \beta^{DA} ) d_{l}^{DA}   + ( \alpha^{RT} (d - \sum_{k\in\mathcal{L}} d_k^{DA}) + \beta^{RT} ) (d_l - d_{l}^{DA} ) \\\nonumber
	&=  &\!\!\! ( \alpha^{DA} \sum_{k\in\mathcal{L}} d_k^{DA} + \beta^{DA} )  d_l  + \alpha^{RT} (d - \sum_{k\in\mathcal{L}} d_k^{DA})(d_l-d_l^{DA})  	\\\nonumber
	&=  &\!\!\! \alpha^{DA} d_l d^{DA}_l   + \alpha^{RT} \sum_{k\in\mathcal{L}\backslash \{l\}} (d_k - d_k^{DA}) (d_l - d_l^{DA})  +  \alpha^{RT} (d_l-d_l^{DA})^2 + \alpha^{DA}  \sum_{k\in\mathcal{L}\backslash \{l\}} d_k^{DA} d_l + \beta^{DA} d_l  ,
\eq
%\end{subequations}
\end{small}%
where the second equality follows from $\beta^{RT} = \lambda^{DA}$.
Given Lemma \ref{lemma:gameeq} and the strict convexity of the expenditure function $c_l(\vec d_l;\vec{d}_{-l})$ in $d_l^{DA}$, the Nash equilibrium of the load-side Cournot game can be characterized by imposing the first-order optimality condition on all the loads, i.e., for $\forall l\in\mathcal{L}$,
\beq\label{eq:gameeq1}
 \alpha^{DA} d_l 
 -   \alpha^{RT} \!\!\!\!\!  \sum_{k\in\mathcal{L}\backslash \{l\}}  \!\!\!\! (d_k-d_k^{DA*})   -  2\alpha^{RT} (d_l-d_l^{DA*})    = 0  , 
\eeq
or equivalently,
\beq\label{eq:gameeq2}
d_l^{DA*} =\left( 1-\frac{\alpha^{DA}}{2\alpha^{RT}} \right)d_l   +  \frac{1}{2}\sum_{k\in\mathcal{L}\backslash \{l\}} (d_k-d_k^{DA*})   .
\eeq
%\hl{Recall} 
Note that the first term on the right-hand side of \eqref{eq:gameeq2} is exactly the individual optimum without any competitors in \eqref{eq:singloadloadminimizer}, while the second term represents the influence of competition. Intuitively, if other loads participate more in the real-time market, load $l$ will increase its participation in the day-ahead market to hedge against rising real-time prices.

Combining \eqref{eq:gameeq2} for all $\l\in\mathcal{L}$ naturally yields the unique solution \eqref{eq:gameeq3}. We can readily observe $d^{DA*}_l > 0$, which is consistent with Lemma \ref{lemma:gameeq}.
The theorem follows. 
%Therefore, \eqref{eq:gameeq3} is a Nash equilibrium where the constraint $d_l^{DA}\ge0, l\in\mathcal{L}$ is not binding. 
%We also need to consider the potential equilibria where the constraint $d_l^{DA}\ge0, l\in\mathcal{L}$ is binding for some $l\in\mathcal{L}$. 
%The idea is to use proof by contradiction. We skip the details here since the main proof technique follows the derivation of \eqref{eq:gameeq2} by the first-order optimality condition.
%In summary, \eqref{eq:gameeq3} is the unique Nash equilibrium of the load-side Cournot game.
\end{proof}

	By summing \eqref{eq:gameeq3} over $\mathcal{L}$ and reorganizing the expression, we are able to derive the following: 
	\begin{corollary}\label{cor:multiloadeq}
		At the Nash equilibrium of the load-side Cournot game, the total day-ahead load and real-time load are respectively 
	\beq\label{eq:gameDAclearedload}
	\begin{split}
	\sum_{l\in\mathcal{L}} d_l^{DA*} &= \left(1 -\frac{\alpha^{DA}}{(L+1)\alpha^{RT}}\right) \sum_{l\in\mathcal{L}} d_l  , \\
	\sum_{l\in\mathcal{L}} d_l^{RT*} &= \frac{\alpha^{DA}}{(L+1)\alpha^{RT}} \sum_{l\in\mathcal{L}} d_l	 ,
	\end{split}
	\eeq	
	which implies $d > d^{DA*}  > \frac{L}{L+1}   d$ and $\frac{1}{L+1}   d >  d^{RT*} > 0 $. 
	Therefore, $\lambda^{RT}> \lambda^{DA}$ and a strictly negative spread follow.
	
	\end{corollary}

\begin{remark}
	Notably, the optimal load participation \eqref{eq:singloadloadminimizer} in the single-load case is a special case of \eqref{eq:gameDAclearedload} where $L=1$. 
	Corollary \ref{cor:multiloadeq} generalizes the conclusion to multi-load cases, and specifically it states that the strategic behavior of load participants even with inelastic demand can reduce market efficiency by taking advantage of the two-stage settlement mechanism.
%		When there are multiple strategic load participants competing in a two-stage settlement electricity market, $\lambda^{RT}> \lambda^{DA}$ still holds at the unique market equilibrium and the nonzero spread indicates the loss of market efficiency by their strategic behavior. 
	However, as the number of load participants $L$ increases, the total day-ahead load approaches the total load and the spread diminishes towards zero, meaning the restoration of market efficiency.
	This is consistent with the intuition that when there are infinite participants, the individual impact on market prices becomes negligible and therefore the market power of each strategic load vanishes, which drives the market to be competitive.
\end{remark}

%	the efficiency loss decrease as more participants are involved.

%Besides, it can be anticipated that as $L$ increases, the day-ahead cleared load approaches the total load and the real-time cleared load finally diminishes when $L$ goes to infinity, thus achieving market efficiency with a zero spread. 

%As we can see, despite that introducing competition relieves the loss of efficiency to some extent, the adverse impact of strategic behavior by load participants remains.

\subsection{The Role of Virtual Bidding}

Virtual bidding is an essential part of competitive electricity markets as it mitigates market power.
Virtual bidders profit from arbitrage on nonzero spreads. 
%They basically involve financial contracts awarded at day-ahead prices and settled at real-time prices. 
%Virtual bidding has been broadly recognized and practiced in the power industry worldwide for potentially improving market efficiency and mitigating market power.
As analyzed above, we have demonstrated that systematic negative spreads can result from the strategic behavior of load participants. However, through an extended analysis of the prior load-side Cournot competition, we now show decrement bids in virtual bidding that act like load participation play an important role in driving these spreads to zeros.

In particular, consider a set $\mathcal{V}:=\{1,2,\dots,V\}$ of virtual bidders. They individually determine their participation $(d_v^{DA},d_v^{RT}), v\in\mathcal{V}$ to compete in the day-ahead and real-time markets in pursuit of arbitrage. However, they differ from real load participants $l\in\mathcal{L}$ in that no real demand needs to be satisfied, i.e., $d_v=0$, $v\in\mathcal{V}$. 
The following theorem characterizes the involvement of these virtual bidders in the load-side Cournot game:
\begin{theorem}
In the two-stage settlement electricity market, there exists a unique Nash equilibrium of the load-side Cournot game with virtual bidders, where the virtual bids are given by
\beq\label{eq:virtualgameeq}
\begin{aligned}
	& d_v^{DA*} =     \frac{\alpha^{DA}}{(L+V+1)\alpha^{RT}}  \sum_{l\in\mathcal{L}} d_l , \\
	& d_v^{RT*} =  -   \frac{\alpha^{DA}}{(L+V+1)\alpha^{RT}}  \sum_{l\in\mathcal{L}} d_l  ,
\end{aligned}
\eeq
for $\forall v\in\mathcal{V}$.
\end{theorem}
Here $d_v^{DA*} >0 $ represents a decrement bid. 
\begin{corollary}
At the Nash equilibrium of the load-side Cournot game with virtual bidders, the total day-ahead load and real-time load are respectively
\begin{small}
\beq\label{eq:virtualgameDAclearedload}
\begin{split}
\sum_{l\in\mathcal{L}} d_l^{DA*}  +  \sum_{v\in\mathcal{V}} d_v^{DA*} &= \left(1 -\frac{\alpha^{DA}}{(L+V+1)\alpha^{RT}}\right)   \sum_{l\in\mathcal{L}} d_l  , \\
\sum_{l\in\mathcal{L}} d_l^{RT*}  +  \sum_{v\in\mathcal{V}} d_v^{RT*} &= \frac{\alpha^{DA}}{(L+V+1)\alpha^{RT}}   \sum_{l\in\mathcal{L}} d_l  ,
\end{split}
\eeq\end{small}%
which implies $d > d^{DA*}  > \frac{L+V}{L+V+1}   d$ and $\frac{1}{L+V+ 1}   d >  d^{RT*} > 0 $. 
As the number of virtual bidders $V$ goes to infinity, the total day-ahead load approaches the total load and the spread converges to zero.
%Therefore, $\lambda^{RT}> \lambda^{DA}$ and a strictly negative price spread follow.
\end{corollary}

%This expression indicates that the introduction of decrement bidders drives the total day-ahead cleared load to approach the total load and therefore contributes to the convergence of the spread to zero. 
%The result is generalizable to virtual generation bidders.
%This demonstrates a win-win situation as stated below:
%The following corollary summarizes this mutually beneficial situation:
\begin{remark}
	Virtual bidders have the incentive to arbitrage over the negative spread resulting from the strategic behavior of load participants, which in turn contributes to alleviating the loss of market efficiency  by driving the two-stage market prices to equalize.
\end{remark}
\begin{remark}\label{rmk:realload}
From \eqref{eq:virtualgameDAclearedload}, the real demand from load participants in the day-ahead market remains positive but actually decreases in the number of virtual bidders $V$, as captured below:
\beq\label{eq:virtualgamerealload}
\sum_{l\in\mathcal{L}} d_l^{DA*}  = \left(1 -\frac{(V+1)\alpha^{DA}}{(L+V+1)\alpha^{RT}}\right)   \sum_{l\in\mathcal{L}} d_l  .
\eeq

\end{remark}

\vspace{0.15in}

\section{Real-World Data Validation}\label{sec:validation}

We next employ real-world electricity market data from NYISO to verify the extent to which our model and analysis reflect real market conditions.

\subsection{Electricity Market Model}

 Day-ahead and real-time loads and prices are collected for the whole year of 2018\footnote{Note that several periods of time, such as Jan. 1-9 and May 21-31, that exhibit extremely abnormal price elasticity of supply are removed.}. Uniform \emph{energy clearing prices} are adopted instead of locational marginal prices since emphasis of our analysis is on the two-stage settlement mechanism rather than the physical constraints of power networks.
 %we observe several periods of time in the year when the day-ahead market cleared at abnormal prices that indicate disparate elasticity of supply, e.g., Dec. 5-12 and Jan. 1-9 with abnormally high prices, and May 21-31 with abnormally low prices. Those days are removed since there should be uncommon factors beyond our model affecting the market clearing.
Fig.~\ref{fig:DA} is a scatterplot of day-ahead prices with respect to day-ahead loads. As \eqref{eq:DAprice} suggests, a day-ahead price should be linear in the corresponding day-ahead load. The linear regression result in Table~\ref{tab:DA} shows that both of the pricing coefficients $\alpha^{DA}$ and $\beta^{DA}$ are statistically significant.
Fig.~\ref{fig:RT} is a scatterplot of negative spreads, i.e., $\lambda^{RT}-\lambda^{DA}$, with respect to real-time loads to justify the connection between the day-ahead and real-time prices, identified in \eqref{eq:spread}. A multiple linear regression of real-time prices on real-time loads and day-ahead prices is carried out yielding the result in Table~\ref{tab:RT}, which confirms that the linearity approximately holds. As analyzed, the coefficient $\gamma$ for day-ahead prices is almost 1 and $\alpha^{RT} > \alpha^{DA}$ is observed. 
However, the proposed model cannot account for the negative intercept $\delta$. 
This could be caused by factors that our analysis neglects, such as strategic generation.

%they should be linearly dependent to reflect 

\begin{figure}[thpb]
	\centering
	\includegraphics[width=0.45\textwidth]{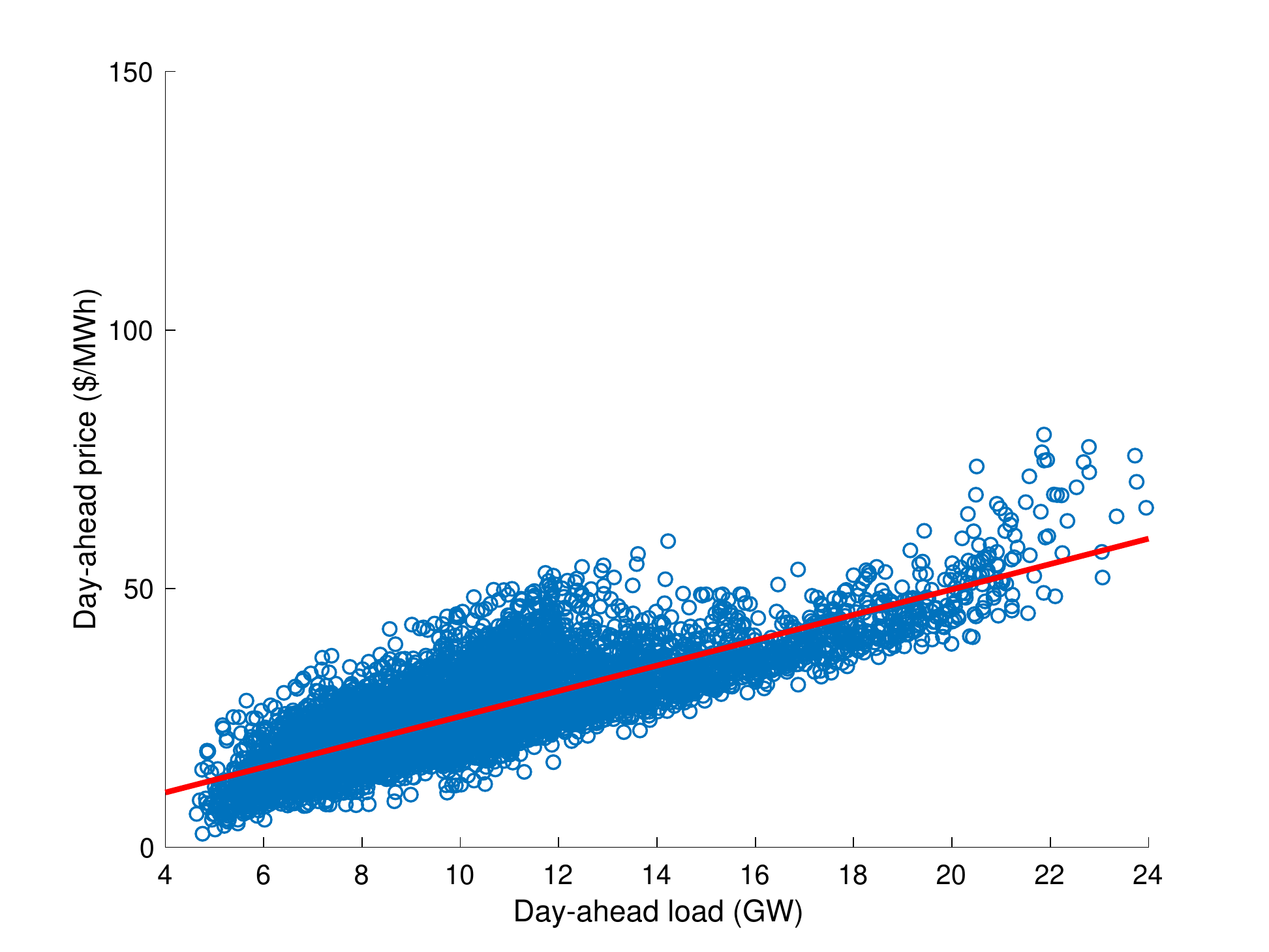}
	\caption{Day-ahead price with respect to day-ahead load, NYISO, 2018.}
	\label{fig:DA}
\end{figure}

\begin{table}[]
	\caption{Linear regression for $\lambda^{DA} = \alpha^{DA} d^{DA} + \beta^{DA}$}
	\label{tab:DA}
	\begin{center}
	\begin{tabular}{|c|c|c|c|c|c|}
		\hline
		& Estimate & Standard error & $p$-value & RMSE & $R^2$  \\ \hline
		 $\alpha^{DA}$ &2.4535  & 0.0208  &  $<$ 0.001  & \multirow{2}{*}{5.7128 } & \multirow{2}{*}{0.6518} \\ \cline{1-4}
		$\beta^{DA}$  &	0.7848 & 0.2253 & 	 $<$ 0.001	&                 &                   \\ \hline
	\end{tabular}
\end{center}
\end{table}

\begin{figure}[thpb]
	\centering
	\includegraphics[width=0.45\textwidth]{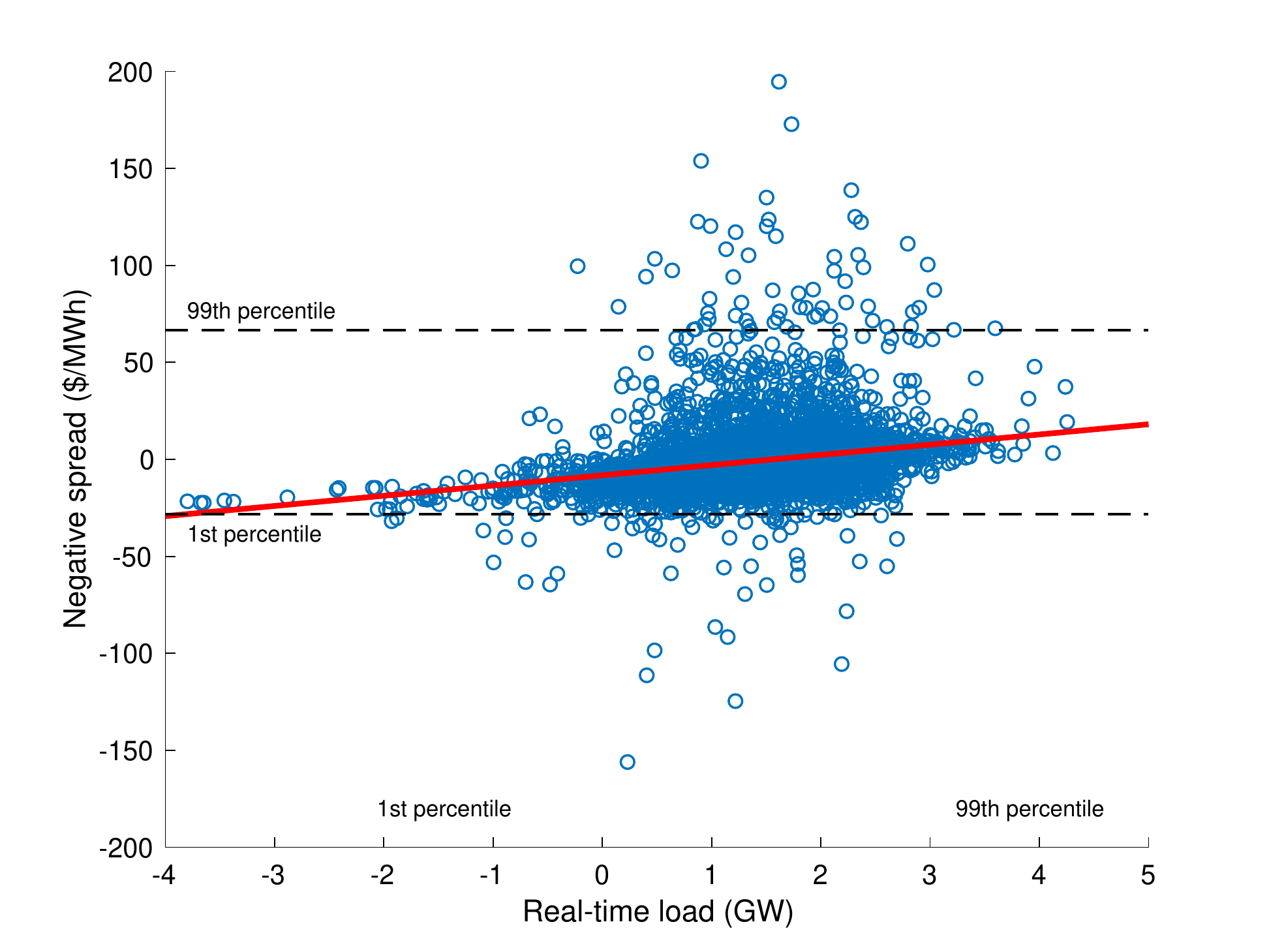}
	\caption{Negative spread with respect to real-time load, NYISO, 2018.}
	\label{fig:RT}
\end{figure}

\begin{table}[]
	\caption{Linear regression for $\lambda^{RT} = \alpha^{RT} d^{RT} + \gamma \lambda^{DA}  + \delta $}
	\label{tab:RT}
	\begin{center}
		\begin{tabular}{|c|c|c|c|c|c|}
			\hline
			 & Estimate & Standard error & $p$-value & RMSE & $R^2$  \\ \hline
			$\alpha^{RT}$ & 5.2658 & 0.1833 &  $<$ 0.001  & \multirow{3}{*}{10.2941 } & \multirow{3}{*}{0.4444} \\ \cline{1-4}
			$\gamma$ &1.0009  & 0.0132 & $<$ 0.001 & &  \\\cline{1-4}
			$\delta$  &	-8.2569 & 0.4980 & 	 $<$ 0.001	&                 &                   \\ \hline
		\end{tabular}
	\end{center}
\end{table}

\subsection{Virtual Bidding}

To validate our analysis of strategic load participants, we assess the special case of virtual bidding due to its significant and verifiable impact on market clearing. The mechanism of virtual bidding was officially introduced into the NYISO market in November, 2001 \cite{NYISOVB}. We collected available data of real loads cleared in the day-ahead market and total actual loads for the several months around that time point to validate the deduction in Remark \ref{rmk:realload}.
It is reasonable to assume $V=0$ prior to the introduction of virtual bidding while $V > 0$ thereafter. As a result, the proportion $1 -\frac{(V+1)\alpha^{DA}}{(L+V+1)\alpha^{RT}}$ is anticipated to diminish with virtual bidding introduced, which is precisely captured by the sudden drop in Fig.~\ref{fig:VB}\footnote{Note that the total load in the NYISO market includes a significant part that is cleared through bilateral transactions outside the market. Therefore, the overall percentage is low. }.

\begin{figure}[thpb]
	\centering
	\includegraphics[width=0.45\textwidth]{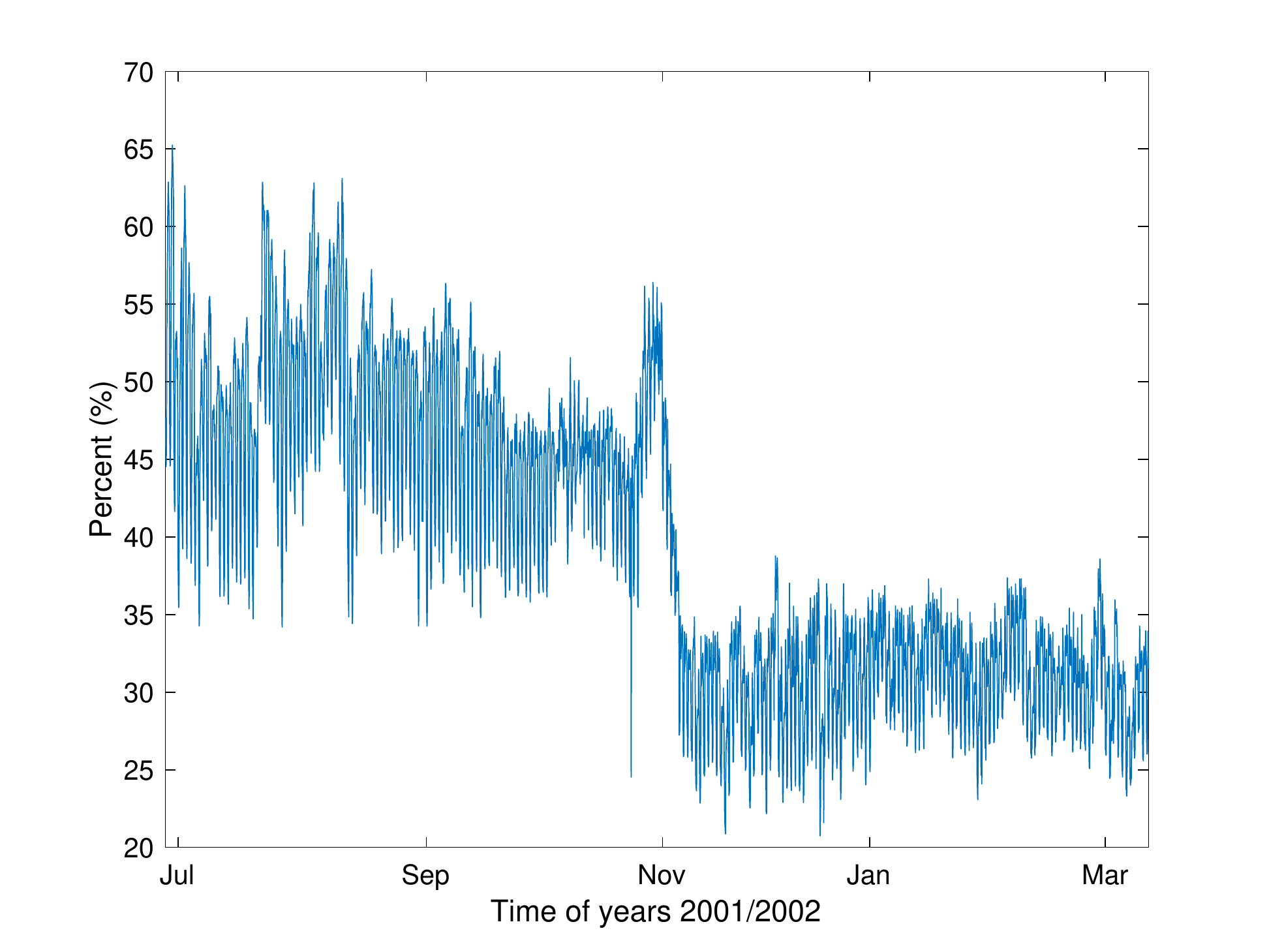}
	\caption{Percentage of day-ahead real load in total load, NYISO, 2001/2002. The introduction of virtual bidding in November, 2001 caused a sudden drop in this percentage.}
	\label{fig:VB}
\end{figure}

%%%%%%%%%%%%%%%%%%%%%%%%%%%%%%%%%%%%%%%%%%%%%%%%%%%%%%%%%%%%%%%%%%%%%%%%%%%%%%%%
%%%%%%%%%%%%%%%%%%%%%%%%%%%%%%%%%%%%%%%%%%%%%%%%%%%%%%%%%%%%%%%%%%%%%%%%%%%%%%%%
%%%%%%%%%%%%%%%%%%%%%%%%%%%%%%%%%%%%%%%%%%%%%%%%%%%%%%%%%%%%%%%%%%%%%%%%%%%%%%%%

\section{Concluding Remarks}\label{sec:conclusion}

This paper develops a model for two-stage settlement electricity markets that explicitly characterizes the interconnection between day-ahead and real-time markets.
Given the model, we attribute systematic negative spreads in electricity markets to the strategic behavior of inelastic load participants that takes advantage of the two-stage settlement mechanism.
We therefore argue that strategic load participation in electricity markets should be taken into account in the characterization of nonzero spreads,  
in addition to empirical factors like load forecast errors or market power of strategic generators. 
Our analysis generalizes to accommodate virtual bidding and demonstrates its role in improving market efficiency by mitigating market power. Real-world market data from NYISO support our theory.

Our model and analysis focus on strategic behavior by inelastic load participants only and are thus not able to account for other factors that can also result in loss of market efficiency. A more comprehensive framework is the subject of ongoing work.

%%%%%%%%%%%%%%%%%%%%%%%%%%%%%%%%%%%%%%%%%%%%%%%%%%%%%%%%%%%%%%%%%%%%%%%%%%%%%%%%
%\section{ACKNOWLEDGMENTS}

%%%%%%%%%%%%%%%%%%%%%%%%%%%%%%%%%%%%%%%%%%%%%%%%%%%%%%%%%%%%%%%%%%%%%%%%%%%%%%%%

\Urlmuskip=0mu plus 1mu\relax
\bibliographystyle{ieeetr}
\bibliography{ref}

%label-appear in arxiv-only one-correspond to lemma in above context
%\begin{comment}

\appendix

\section*{Proof of Lemma \ref{lemma:gameeq}}\label{apx:lemma}

\begin{proof}
	We prove this lemma by contradiction.
	Assume there exists a Nash equilibrium where $d_l^{DA*}=0$ for some $l\in \mathcal{L}$. 
	Define $\mathcal{L}'$ as the set of the remaining loads with $d_l^{DA*}>0$ and let $L':=|\mathcal{L}'|$. 
	Note that the expenditure function $c_l(\vec d_l;\vec{d}_{-l})$ of each load $l\in\mathcal{L}$ can be rewritten as a strictly convex function in terms of $d^{DA}_l$ only; see \eqref{eq:expenditurefunc}.
	For each load $l\in\mathcal{L}\backslash \mathcal{L}'$ with $d_l^{DA*}=0$, the first-order optimality condition does not hold, i.e.,
	\beq\label{eq:gameeqproof3}
		d_l^{DA*} = 0
		\ge  ( 1-\frac{\alpha^{DA}}{2\alpha^{RT}} )d_l   +  \frac{1}{2}  \sum_{k\in\mathcal{L}\backslash \{l\}} (d_k-d_k^{DA*}) 
		=  \frac{\alpha^{RT}- \alpha^{DA}}{2\alpha^{RT}} d_l   +  \frac{1}{2} \sum_{k\in\mathcal{L}}  d_k-\frac{1}{2}\sum_{k\in\mathcal{L}'} d_k^{DA*}  .
	\eeq
	
	However, the first-order optimality condition holds for the loads $l\in\mathcal{L}'$, which can be expressed as 
	\beq\label{eq:gameeqproof1}
	d_l^{DA*} = \frac{\alpha^{RT}- \alpha^{DA}}{2\alpha^{RT}} d_l   +  \frac{1}{2} \sum_{k\in\mathcal{L}}  d_k-\frac{1}{2}\sum_{k\in\mathcal{L}'\backslash   \{l\}} d_k^{DA*}  ,
	\eeq
	where we replace $\mathcal{L}$ with $\mathcal{L}'$ since $d_l^{DA*}=0,l\in\mathcal{L}\backslash \mathcal{L}'$.
	Summing \eqref{eq:gameeqproof1} over $\mathcal{L}'$ and reorganizing the expression lead to
	\beq\label{eq:gameeqproof2}
		\sum_{l\in \mathcal{L}'} d_l^{DA*} 
		=  \left(1 -\frac{\alpha^{DA}}{(L'+1)\alpha^{RT}}\right)\! \sum_{l\in\mathcal{L}'} d_l  +  \frac{L'}{L'+1} \sum_{l\in\mathcal{L}\backslash \mathcal{L}' }d_l 
		\ < \  \sum_{l\in\mathcal{L}} d_l  -   \frac{\alpha^{DA}}{(L'+1)\alpha^{RT}} \sum_{l\in\mathcal{L}'} d_l  \ < \    \sum_{l\in\mathcal{L}} d_l  .
	\eeq
	
	Substituting \eqref{eq:gameeqproof2} into \eqref{eq:gameeqproof3} yields a contradiction:
	\beq\label{eq:gameeqproof4}
	0 \ge \frac{\alpha^{RT}- \alpha^{DA}}{2\alpha^{RT}} d_l   +  \frac{1}{2} \sum_{k\in\mathcal{L}}  d_k-\frac{1}{2}\sum_{k\in\mathcal{L}'} d_k^{DA*}  > 0 ,
	\eeq
	where the second inequality also uses $\alpha^{RT} > \alpha^{DA}$; recall Fig.~\ref{fig:pricecorrelation}. The preliminary assumption is rejected and therefore the lemma is proved.	
\end{proof}

%\end{comment}

%\begin{thebibliography}{99}
%
%\bibitem{c1}
%J.G.F. Francis, The QR Transformation I, {\it Comput. J.}, vol. 4, 1961, pp 265-271.
%
%\bibitem{c2}
%H. Kwakernaak and R. Sivan, {\it Modern Signals and Systems}, Prentice Hall, Englewood Cliffs, NJ; 1991.
%
%\bibitem{c3}
%D. Boley and R. Maier, "A Parallel QR Algorithm for the Non-Symmetric Eigenvalue Algorithm", {\it in Third SIAM Conference on Applied Linear Algebra}, Madison, WI, 1988, pp. A20.
%
%\end{thebibliography}

\end{document}